\documentclass[a4paper,twoside,11pt,reqno]{amsart}   
\usepackage{amssymb,amsfonts,amsmath,stmaryrd,bbm}
\usepackage{latexsym}
\usepackage{gensymb}
\usepackage{verbatim} 
\usepackage{colortbl}
\usepackage{soul}
\usepackage{arydshln}  
\usepackage[utf8]{inputenc}

\usepackage{kpfonts}
\usepackage{afterpage}
\usepackage{dsfont}
\usepackage{color} 
\usepackage{a4wide} 
\usepackage{url}
\usepackage{pifont, tikz, subfigure}
\usetikzlibrary{plotmarks,shapes,arrows,positioning}

\usepackage[plainpages=false,pdfpagelabels,colorlinks=true,citecolor=blue,hypertexnames=false]{hyperref}

\newcommand{\LandauO}{\mathcal{O}}

\newcommand{\sC}{\s{C}}

\newcommand{\sF}{\s{F}}

\newcommand{\sH}{\s{H}}
\newcommand{\sI}{\s{I}}
\newcommand{\sbI}{\s{\overline I}}
\newcommand{\sbJ}{\s{\overline J}}
\newcommand{\sbK}{\s{\overline K}}
\newcommand{\sbZ}{\s{\overline Z}}
\newcommand{\sbU}{\s{\overline U}}
\newcommand{\sJ}{\s{J}}

\newcommand{\sU}{\s{U}}
\newcommand{\sZ}{\s{Z}}
\newcommand{\hI}{\hat \sI}
\newcommand{\hJ}{\hat \sJ}

\newcommand{\cov}{\triangleleft}
\newcommand{\vide}{\varnothing}
\newcommand{\e}[1]{\mathcal{#1}}
\newcommand{\s}[1]{\mathbf{#1}}
\newcommand{\eD}{\mathcal D}
\newcommand{\ebI}{\overline{\e I}}
\newcommand{\ebJ}{\overline{\e J}}
\newcommand{\ebK}{\overline{\e K}}

\newcommand{\nablam}{\nabla_m}

\catcode`\@=11
\def\section{\@startsection{section}{1}%
 \z@{.7\linespacing\@plus\linespacing}{.5\linespacing}%
 {\normalfont\Large\bfseries\scshape\centering}}

\def\subsection{\@startsection{subsection}{2}%
  \z@{.5\linespacing\@plus\linespacing}{.5\linespacing}%
  {\normalfont\large\bfseries\scshape}}

\def\subsubsection{\@startsection{subsubsection}{3}%
 \z@{.5\linespacing\@plus\linespacing}{-.5em}
 {\normalfont\large\bfseries}}
\catcode`\@=12

\newtheorem{theorem}{Theorem}[section]
\newtheorem{lemma}[theorem]{Lemma}
\newtheorem{conjecture}[theorem]{Conjecture}

\newtheorem{prop}[theorem]{Proposition}

\newtheorem{definition}[theorem]{Definition}
\theoremstyle{definition}

\newtheorem{example}[theorem]{Example}

\newcommand{\beq}{\begin{equation}}
\newcommand{\eeq}{\end{equation}}

\renewcommand{\mp}{m^+}

\newcommand{\ns}{\mathbb{N}}

\newcommand{\qs}{\mathbb{Q}}

\newcommand{\Sn}{\mathfrak S}

\newcommand{\fps}{formal power series}
\newcommand{\gf}{generating function}
\newcommand{\gfs}{generating functions}
\def\emm#1,{{\em #1}}

\begin{document}
\date{\today}
%
\title
{Intervals in the greedy Tamari posets}

\author[M. Bousquet-Mélou]{Mireille Bousquet-Mélou}

\thanks{MBM was partially supported by the ANR projects DeRerumNatura (ANR-19-CE40-0018) and Combiné (ANR-19-CE48-0011). FC was partially supported by the ANR projects Charms (ANR-19-CE40-0017) and Combiné (ANR-19-CE48-0011).}

\address{MBM: CNRS, LaBRI, Université de Bordeaux, 351 cours de la
  Libération,  F-33405 Talence Cedex, France\\
FC: Institut de Recherche Mathématique Avancée, UMR 7501 Université de Strasbourg and CNRS, 7 rue René-Descartes, 67000 Strasbourg, France} 
\email{bousquet@labri.fr, chapoton@unistra.fr}

\author[F. Chapoton]{Frédéric Chapoton}

\keywords{Tamari posets --- Planar maps --- Enumeration --- Algebraic generating functions}

\makeatletter
\@namedef{subjclassname@2020}{%
  \textup{2020} Mathematics Subject Classification}
\makeatother

\subjclass[2020]{Primary 05A15  -- Secondary 06A07, 06A11}

\begin{abstract}
  We consider a greedy version of the $m$-Tamari order 
  defined on  $m$-Dyck paths, recently introduced by Dermenjian. Inspired by
  intriguing connections between intervals in the ordinary {1-}Tamari
  order and planar triangulations, and more generally by the existence of simple formulas counting intervals in the ordinary $m$-Tamari orders, we investigate the number of
  intervals in the greedy order on $m$-Dyck paths of fixed size. We find again a simple formula, which also counts
certain planar
  maps (of prescribed size) called $(m+1)$-constellations.

  For instance, when $m=1$ the number of intervals in the greedy
  order on {1-}Dyck paths of length $2n$ is proved to be
  $\frac{3\cdot 2^{n-1}}{(n+1)(n+2)} \binom{2n}{n}$, which is also the
  number of bipartite maps with $n$ edges.

  Our approach is recursive, and uses a ``catalytic'' parameter, namely the length of the final descent of the upper path of the interval. The resulting bivariate \gf\ is algebraic for all $m$. We show that the same approach can be used to count intervals in the ordinary $m$-Tamari lattices as well. We thus recover 
  the earlier result of the first author, Fusy and Préville-Ratelle, who were using a different catalytic parameter.
\end{abstract} 

\maketitle

\section{Introduction}
In the past 15 years, several intriguing connections have emerged between  intervals in various posets defined on lattice paths, and families of  planar maps~\cite{BeBo07,ch06,mbm-fusy-preville,mbm-chapuy-preville,duchi-henriet,fang-bipartite,fang-bridgeless,fang-preville-enumeration,fang-trinity,chapoton-dexter}. This paper adds a picture to this gallery by establishing an enumerative link between \emm intervals in  greedy $m$-Tamari posets, and \emm planar constellations,.

Let us fix an integer $m \geq 1$. An $m$-Dyck path (or Dyck path, for short) of size $n$ is a sequence of $n$ up steps $(+m,+m)$ and $m n$ down steps $(+1,-1)$
starting from $(0,0)$ in $\ns^2$ and never going strictly below the
horizontal axis. The set of such
paths is
denoted by $\eD_{m,n}$.  We {encode} 
Dyck paths by words, using the letter $1$ for up steps and $0$ for down steps.  A \emm factor, of a path/word $w$ is a sequence of consecutive steps/letters. A
{\it valley} of $w$ is an occurrence of a factor $01$.

\begin{figure}[tbh]
  \centering
     	\scalebox{0.85}{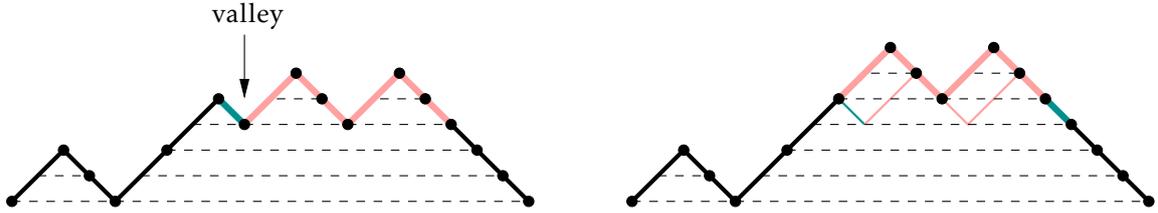}
  \caption{Example of greedy cover relation, with $n=5$ and {$m=2$}.}
  \label{fig:greedy} 
\end{figure}

The  greedy Tamari partial order on the set $\eD_{m,n}$ was introduced in~\cite{aram2022}.  It is defined by its cover
relations, as follows. For every Dyck path $w$, there is a cover
relation $w \cov w'$ for every valley of $w$. The path
$w'$ is obtained by swapping
the down step of the valley and
the longest Dyck factor that follows it.
See Figure~\ref{fig:greedy} for an example.
Recall that in the \emm ordinary, Tamari order, cover relations are obtained by swapping the down step of the valley and
the \emm shortest, Dyck factor that follows it~\cite{bergeron-preville,mbm-fusy-preville}. Hence  cover relations in the greedy order correspond to certain sequences of cover relations in the ordinary order. In particular, any greedy interval $[v,w]$ is also an interval in the ordinary Tamari poset.  We refer to Figures~\ref{fig:poset-n=3} and~\ref{fig:poset} for comparison of the two posets, in the case where $m=1$ and $n=3$ or $4$.

Our main result gives the number of greedy intervals in $\eD_{m,n}$.
 
\begin{figure}[bht] 
  \centering
  \includegraphics[height=35mm]{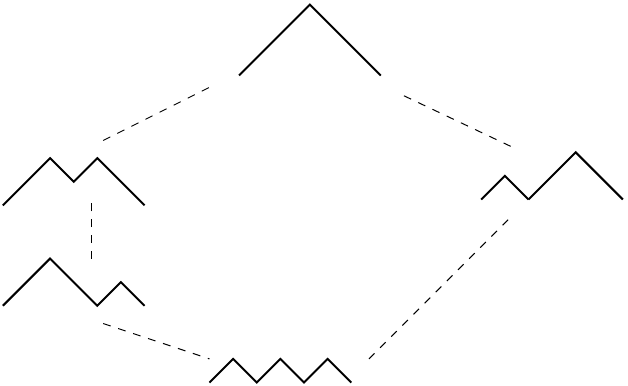}\hskip 10mm
   \includegraphics[height=35mm]{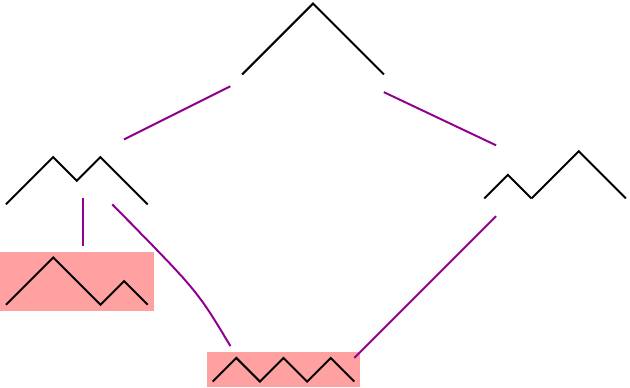}
  \caption{The ordinary Tamari lattice on $\eD_{1,n}$ with $n=3$ (left), and its greedy counterpart (right). The shaded elements are minimal for the greedy order. The cover relations in the greedy order realize shortcuts in the ordinary Tamari lattice.}
  \label{fig:poset-n=3}
\end{figure}
 
\begin{figure}[bht] 
  \centering
  \includegraphics[height=60mm]{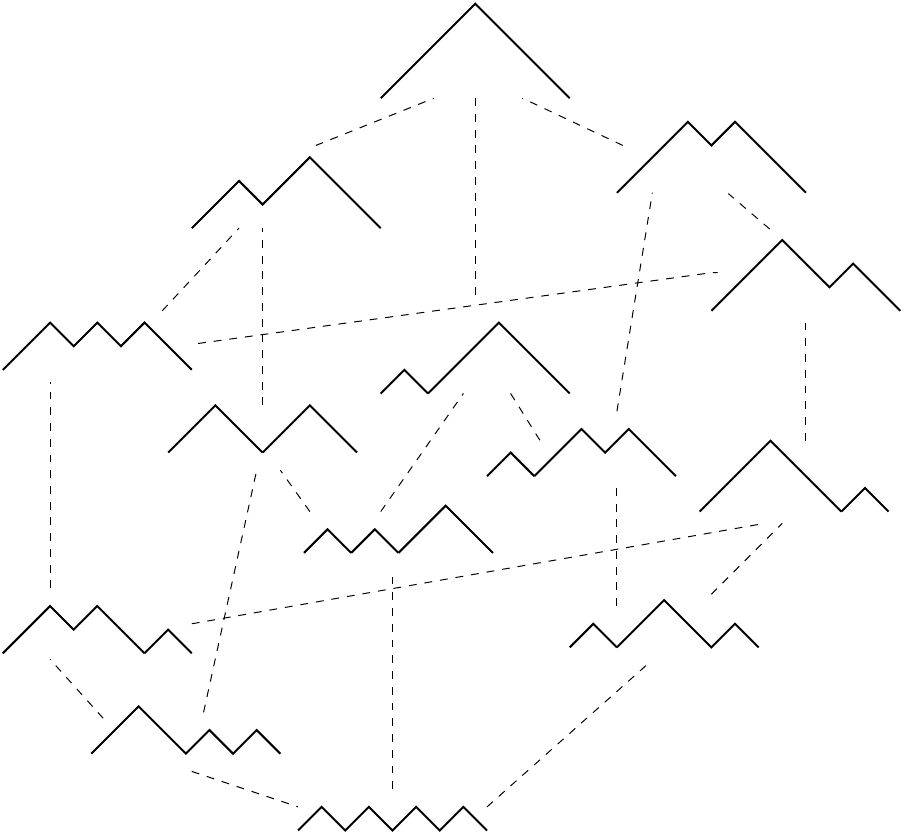}\hskip 10mm
   \includegraphics[height=60mm]{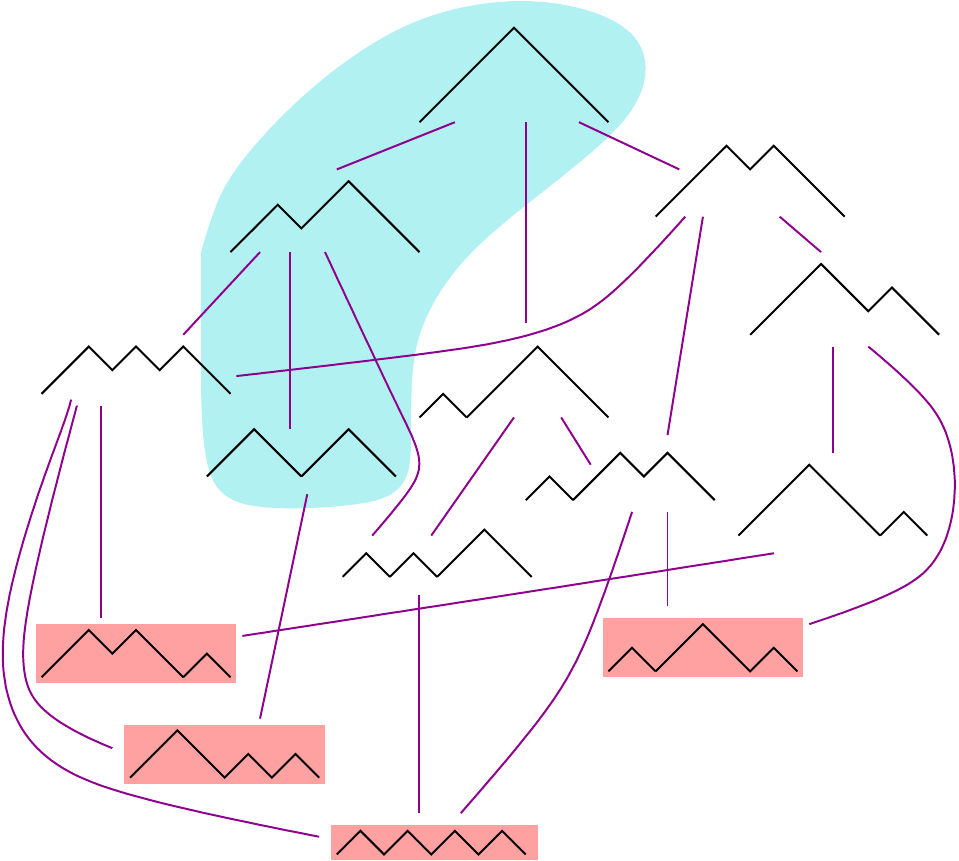}
   \caption{The ordinary Tamari lattice on $\eD_{1,n}$ with $n=4$ (left), and its greedy counterpart (right).
      The chain above $11001100$ is order-isomorphic to the greedy order on $\eD_{2,2}$.}
  \label{fig:poset}
\end{figure}

\begin{theorem}\label{thm:numbers}
  The number of intervals in the greedy $m$-Tamari poset $\eD_{m,n}$ is
  \[
    \frac{(m+2) (m+1)^{n-1}}{(mn+1)(mn+2)}\binom{(m+1)n}{n}.
  \]
  This is also the number of \emm $(m+1)$-constellations, having $n$ polygons.
\end{theorem}

Let us briefly recall from~\cite{mbm-schaeffer-constellations,lando-zvonkine-book} that an $m$-constellation is a rooted planar map (drawn on the sphere) whose faces are coloured black and white in such a way that
\begin{itemize}
\item all faces adjacent to a  white face are black, and vice-versa,
\item the degree of any black face is $m$,
\item the degree of any white face is a multiple of $m$.
\end{itemize}
Constellations are rooted by distinguishing  one of their edges. In Theorem~\ref{thm:numbers}, what we call \emm polygon, is a black face.  We refer to~\cite{mbm-schaeffer-constellations} for an alternative description of constellations in terms of factorisations of permutations and  for their bijective enumeration, to~\cite{bouttier-mobiles} for an alternative bijective approach, and to~\cite[Sec.~4.2]{fang-these} for a recursive approach.

\begin{example}\label{ex:simple}
When $m=n=2$, there are three $m$-Dyck paths of size $n$, namely
\[
  u=100100, \quad v=101000, \quad w=110000,
\]
and the greedy $2$-Tamari order is the total order $u<v<w$. Hence there are $6$ greedy intervals, and we can check that the formula of Theorem~\ref{thm:numbers} holds:
  \[
    \frac{(2+2) (2+1)^{1}}{(2\cdot 2+1)(2\cdot 2 +2)}\binom{(2+1)2}{2}=6.
  \]
The corresponding $6$ constellations with $n=2$ polygons (triangles, since $m+1=3$) are shown in Figure~\ref{fig:const}.

\begin{figure}[htb]   
\begin{center}
\includegraphics[height=20mm]{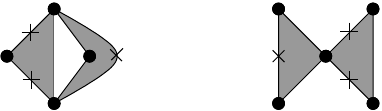} 
\end{center} 
\caption{The six  $3$-constellations with $2$ polygons. The crosses indicate the possible rootings.}
\label{fig:const}     
\end{figure}
\end{example}

\begin{example} Let us now take $m=1$ and $n=3$, and consider the greedy poset shown on the right of  Figure~\ref{fig:poset-n=3}: if we count, for all paths $u$ (taken from bottom to top and from left to right), the number of intervals of the form $[u,v]$, we find that the total number of intervals~is
  \[
    4+ 3 + 2 + 2 + 1 = 12 =   \frac{(1+2) (1+1)^{2}}{(3+1)(3+2)}\binom{2 \cdot 3}{3}.
  \]
  This is one less than the number of intervals in the ordinary Tamari lattice, shown on the left of Figure~\ref{fig:poset-n=3}.
\end{example}

\medskip

More generally, the number of intervals of $\eD_{m,n}$ in the ordinary $m$-Tamari order is~\cite{mbm-fusy-preville,ch06}:
\beq\label{ordinary}
\frac {m+1}{n(mn+1)} {(m+1)^2 n+m\choose n-1}.
\eeq
These numbers first arose, conjecturally, in the study of polynomials in three sets of $n$ variables on which the symmetric group $\Sn_n$ acts diagonally~\cite{bergeron-preville}. We are not aware of any occurrence of the numbers of Theorem~\ref{thm:numbers} in such an algebraic context.

\medskip
Our proof of Theorem~\ref{thm:numbers} follows  a recursive approach, similar to the proof of~\eqref{ordinary} in~\cite{mbm-fusy-preville}. We begin in Section~\ref{sec:defs} with various definitions and constructions on Dyck paths and greedy Tamari intervals. In Section~\ref{sec:dec}, we use them to write a functional equation for the \gf\ of these intervals (Proposition~\ref{prop:func-eq}). This \gf\ records not only the \emm size, $n$ of the interval $[v,w]$ (defined as the common size of $v$ and $w$), but also  the length of the longest suffix of down steps in $w$, called \emm final descent, of $w$. This parameter plays a ``catalytic'' role, in the sense that the functional equation cannot be written without it. Equations involving a catalytic variable abound in the enumeration of maps, and their solutions are systematically algebraic~\cite{mbm-jehanne}.  The enumeration of \emm ordinary, Tamari intervals also relies on a similar equation~\cite{mbm-fusy-preville}. In Section~\ref{sec:sol}, we solve our functional equation and express the bivariate \gf\ of intervals in terms of a pair of algebraic series (Theorem~\ref{thm:series}). In Section~\ref{sec:sol-ord} we prove that our approach applies to \emm ordinary, $m$-Tamari intervals as well, and derive a new proof of~\eqref{ordinary}, thus recovering the result of~\cite{mbm-fusy-preville}. We conclude in Section~\ref{sec:final} with a few comments and open questions. In particular, we conjecture the existence of a bijection between greedy intervals and constellations that would transform ascents of the top path into degrees of white faces (Conjecture~\ref{conj}).

\medskip 
As a side remark, our results were inspired by those obtained about
the number of intervals in the so-called \emm dexter semilattices,
in~\cite{chapoton-dexter}. In fact, it is expected that the $m=1$ greedy
Tamari posets are anti-isomorphic to the dexter posets, through a
simple bijection that goes from Dyck paths to binary trees, performs a
left-right-symmetry there and then comes back to Dyck paths by the
same bijection.

\section{$m$-Dyck paths and greedy partial order}
\label{sec:defs}
Let us fix $m\ge1$. We first complete the definitions introduced in the previous section.

The {\it height} of a vertex on an ($m$-)Dyck path is the $y$-coordinate of
this vertex. A  {\it contact} is a vertex of height $0$, distinct from the endpoints.  A {\it peak} is an occurrence of a factor $10^m$.
The unique Dyck path of size zero is called the empty Dyck path and
denoted by $\vide$.

The \emm final descent, of a Dyck path $w$ is the longest suffix of the form $0^k$, and we denote its length  by $d(w)=k$. If $w$ is non-empty, then $k\ge m$. Thus $w$ has at least one peak, consisting of the last up step and the $m$ down steps that follow it.

It is well known that every non-empty Dyck path $w$ admits a unique expression of the form $w=   1(w_1 0)(w_2 0)\cdots(w_m 0)w_{m+1}$,  where $w_1,\ldots,w_{m+1}$ are Dyck paths, possibly empty~\cite[Sec.~11.3]{lothaire1}.
We denote by $D(w_1, \ldots, w_m, w_{m+1})$ the Dyck path $1(w_1 0)(w_2 0)\cdots(w_m 0)w_{m+1}$.

\smallskip

Let us now return to the greedy Tamari order defined in the introduction, and to its ordinary counterpart. For both orders, if $v\le w$ then $v$ lies below $w$, in the sense that the height of the $i$th vertex of $v$ is at most the height of the $i$th vertex of $w$, for any $i$.

Moreover, the set $\eD_{m,n}$ equipped with the greedy $m$-Tamari order is order-isomorphic to an upper ideal of $\eD_{1,mn}$ equipped with
the greedy $1$-Tamari order.
  The same statement was already true for the ordinary Tamari order~\cite[Prop.~4]{mbm-fusy-preville}. 
  \begin{prop}
    The greedy $m$-Tamari order on $\eD_{m,n}$ is order-isomorphic to the greedy $1$-Tamari order on the upper ideal of $\eD_{1,mn}$ consisting of paths in which the length of every ascent is a multiple of $m$ (by  \emm ascent,, we mean a maximal sequence of up steps).
  \end{prop}
  As in the ordinary case,  the proof simply consists in replacing each up step of height $m$ in a path of $\eD_{m,n}$ by a sequence of $m$ up steps of height $1$. The key property is that a cover relation may merge two ascents, but never splits an ascent into several parts. For instance, on the right of Figure~\ref{fig:poset}, we recognize in the upper ideal generated by $11001100$ the greedy Tamari lattice on $\eD_{2,2}$ described in Example~\ref{ex:simple}.

  \subsection{A free monoid structure on $m$-Dyck paths}

In the enumeration of ordinary $m$-Tamari intervals~\cite{mbm-fusy-preville}, a useful property is the fact that, for two Dyck paths $v$ and $w$ of the same length, with $w=w_1 w_2$ written as the concatenation of two Dyck paths, we have $v\le w$ if and only if $v=v_1v_2$ for two Dyck paths $v_1$ and $v_2$ such that $v_1\le w_1$ and $v_2\le w_2$. This leads
to a recursive decomposition of ordinary intervals using the number of contacts (of the smaller element) as a catalytic variable. This property does not hold in the greedy case: for instance, when $m=1$ and $n=3$, the path $w=110010$ is not larger than $v=101010$, even though $w=w_1w_2$ and $v=v_1v_2$ with $v_1=1010< 1100=w_1$ and $v_2=w_2=10$. This problem may not be irremediable, but in this paper we take another route. In fact, our approach also gives a new solution for ordinary $m$-Tamari intervals.

We define on  $\eD_m:=\cup_{n\ge 1} \eD_{m,n}$ a new product, different from concatenation, and show that it endows $\eD_m$ with a structure of graded free monoid. In the next subsection, we will see that this new product is, in a sense, compatible with the greedy order.

Let $w_1$ and $w_2$ be non-empty Dyck paths in $\eD_m$. The product $w_1 * w_2$ is defined as
the Dyck path obtained
by replacing the rightmost peak of 
$w_1$ by $w_2$ (Figure~\ref{fig:product-paths}). One can check that this is associative, with unit the Dyck path
$10^m$ made of a single peak. This product is moreover graded with degree the size minus $1$. Indeed,  if we denote the size of $w$ by $|w|$, we have $|w_1 * w_2| = |w_1|+ |w_2|-1$, that is, $|w_1 * w_2|-1 = (|w_1|-1) + (|w_2|-1)$.
  
\begin{figure}[htp]
  \centering
  	\scalebox{0.9}{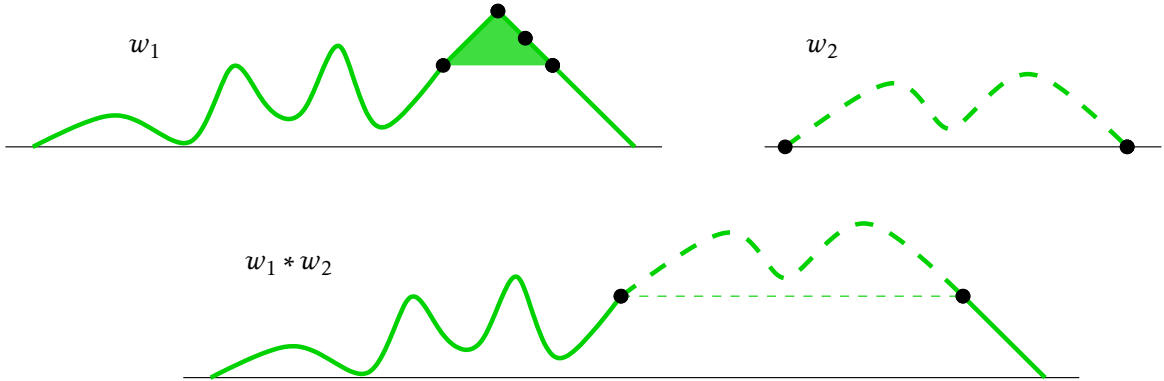}
        \caption{The monoid structure on $m$-Dyck paths, when $m=2$.}
        \label{fig:product-paths} 
      \end{figure}

Here is an example with $m=2$, which shows that a path admits in general several expressions as a product:
\begin{equation*}
  1100{\color{red}100}00 * 110000 
  = 1{\color{red}100}00 * 100110000 = 110011000000.
\end{equation*}
This is because some of the factors can be further decomposed as products.
In fact, we have the following unique factorisation property.

\begin{prop}\label{prop:monoid_words}
  The set $\eD_m$ equipped  with the product $*$ is a free graded monoid on the 
  generators
$    D(w_1,w_2,\dots,w_{i-1},10^m,\vide,\dots,\vide)$,  with $1\leq i \leq m+1$ and  $w_j \in \eD_m \cup \{\vide\}$ for all $j$.
\end{prop}

\begin{proof}
  Let us first show that $\eD_m$ is generated by these generators. Let $w$
  be a non-empty Dyck path.
  Write $w=D(w_1,w_2,\dots,w_m, w_{m+1})$. If
  all $w_i$'s are empty, then $w$ is the unit $10^m$. Otherwise, let $i$
  be the largest index such that $w_{i}$ is non-empty, with $1\le i \le m+1$. Then
  $w = D(w_1,\dots,w_{i-1},10^m,\vide,\dots,\vide) * w_{i}$. This is the product of a generator by an element of~$\eD_m$, namely $w_{i}$, of smaller degree than  $w$.
  Thus $\eD_m$ is   generated by the listed  generators.

 For instance, when $m=2$, the above path $w=110011000000$ can be iteratively factored~as:
    \begin{align*}
      w &= 110000 * 100110000\\
        &= 110000 * 100100 * 110000\\
        &= D(100, \vide, \vide) *  D(\vide, \vide, 100)* D(100, \vide, \vide).
    \end{align*}
 
  We will now   use comparison of generating functions
  to prove that the   monoid is free. The generating function of $m$-Dyck paths, where the variable $t$ records the size, is the only \fps \ $\s D$ in the variable $t$ satisfying
  \begin{equation*}
    \s D = 1 + t \s D^{m+1}.
  \end{equation*}
  This follows from the  decomposition of Dyck paths as $D(w_1,w_2,\dots,w_m,w_{m+1})$.
  The generators are counted by 
  \begin{equation*}
   \s G :=  t^2 \sum_{i=1}^{m+1} \s D^{i-1}= t^2\frac{\s D^{m+1}-1}{\s D-1}
    =t - \frac{t^2}{\s D-1}.
  \end{equation*}
 Recall that the path $w_1 * \cdots * w_k$ has size $|w_1|+ \cdots + |w_k|-k+1$. Hence
  the free monoid on the listed  generators has size  generating function
  \[
    \frac t{1-\s G/t} = \s D-1,
  \]
  which coincides with the \gf\ of $\eD_m$. This implies that there are no relations between the generators. 
\end{proof}

Let us extract from the above proof an observation that  will be useful later.
\begin{lemma}
  \label{lem:first_factor}
  If $w=D(w_1,\dots,w_{i-1},w_{i},\vide,\dots,\vide)$ with $w_{i}\not = \vide$, then the first factor in the factorisation of $w$ is $D(w_1,\dots,w_{i-1},10^m,\vide,\dots,\vide)$.
\end{lemma}

\subsection{A free monoid structure on intervals}

The monoid structure $*$ on $\eD_m$ is compatible with the greedy Tamari
order, in the following sense.

\begin{prop}
  \label{covering_and_product}
  Let $v = v_1 * v_2$ be a non-empty Dyck path. Let $v \cov w$ be a
  greedy cover relation.  Then either $w = w_1 * v_2$ where $w_1$ covers $v_1$, or $w=v_1 * w_2$ where $w_2$ covers $v_2$.
 
  Conversely, every cover relation $v_1 \cov w_1$ gives a
  cover relation $v_1 * v_2 \cov w_1 * v_2$ and every cover relation
  $v_2 \cov w_2$ gives a cover relation $v_1 * v_2 \cov v_1 * w_2$.

 Consequently, for any non-empty Dyck path $v = v_1 * v_2$, the upper ideal $\{w: v \leq w\}$ is $\{w_1 *w_2: v_1\le w_1 \text{ and } v_2\le w_2\}$.
\end{prop}

\begin{proof}
  Let us begin with the first statement. Consider the down step of   the valley of $v$ involved
  in the cover relation $v \cov w$. If this down step  comes from 
  $v_1$ (as in Figure~\ref{fig:compatibility1}), it is the first step of a valley  of $v_1$, and hence
  defines a cover relation $v_1 \cov w_1$. This cover relation may or may not
  increase  the height of the rightmost peak of~$v_1$ (in the example of Figure~\ref{fig:compatibility1} this height increases). In both cases, by
  definition of the product, one finds that $w_1 * v_2 = w$.

  \begin{figure}[htp]
  \centering
  	\scalebox{0.9}{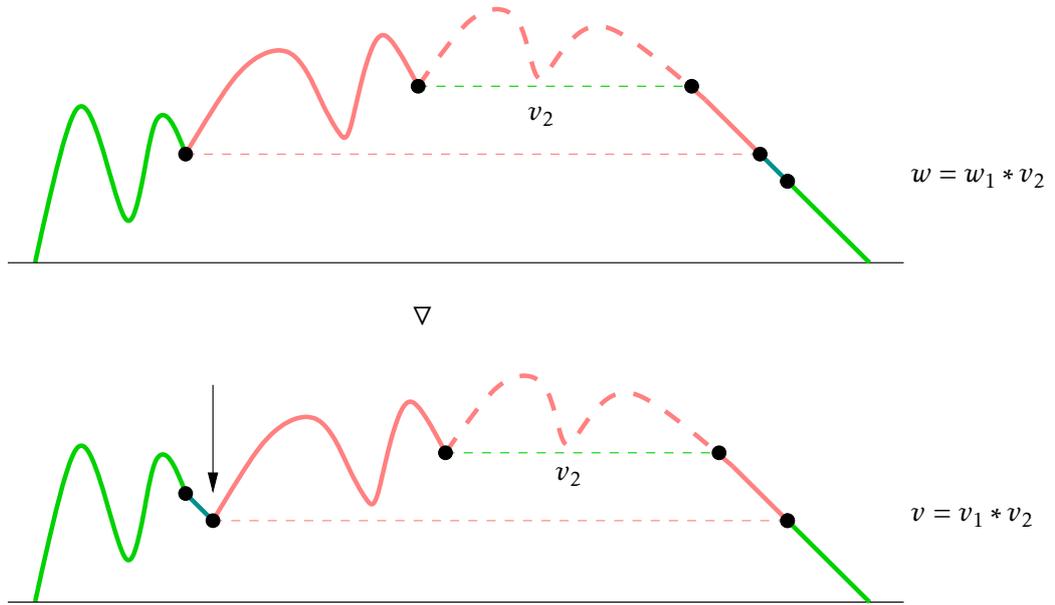}
        \caption{First type of cover relation $v\cov w$ when $v=v_1 *v_2$. The cover relation involves a down step of $v_1$.}
        \label{fig:compatibility1} 
      \end{figure}
 
  If the down step involved in the cover relation $v \cov w$ comes instead from 
  $v_2$ (Figure~\ref{fig:compatibility2}), then the  next up step also comes from $v_2$, hence it defines a cover relation
  $v_2 \cov w_2$. By definition of the product, one finds 
  that   $v_1 * w_2 = w$. 

   \begin{figure}[htp]
  \centering
  	\scalebox{0.9}{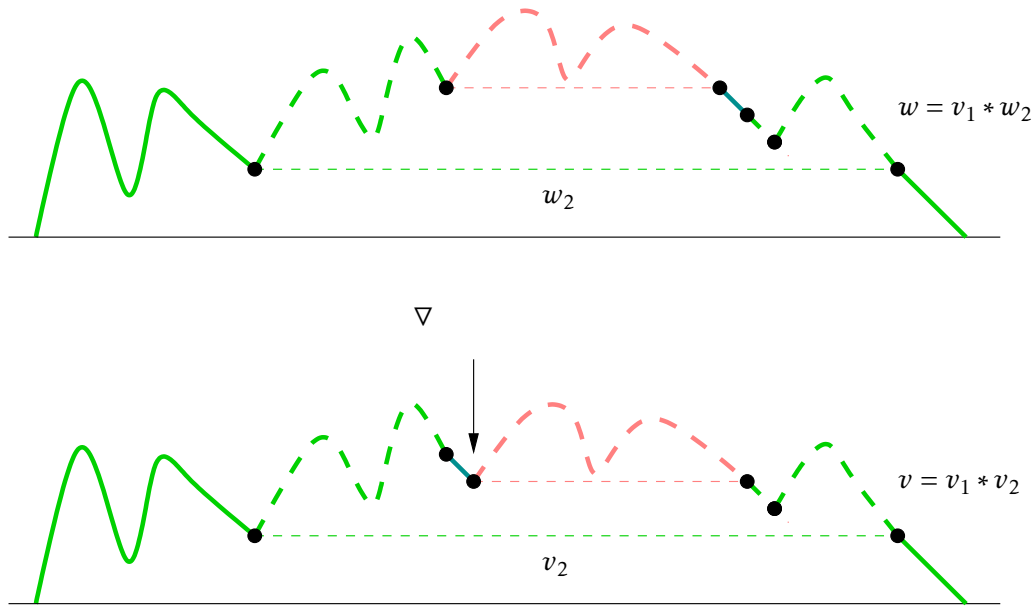}
        \caption{Second type of cover relation $v\cov w$ when $v=v_1 *v_2$. The cover relation involves a down step of $v_2$.}
        \label{fig:compatibility2} 
      \end{figure}

  The second   statement is checked similarly from the definitions, and the third one follows by iteration of cover relations.
\end{proof}

\begin{figure}[htp]
  \centering
  	\scalebox{0.9}{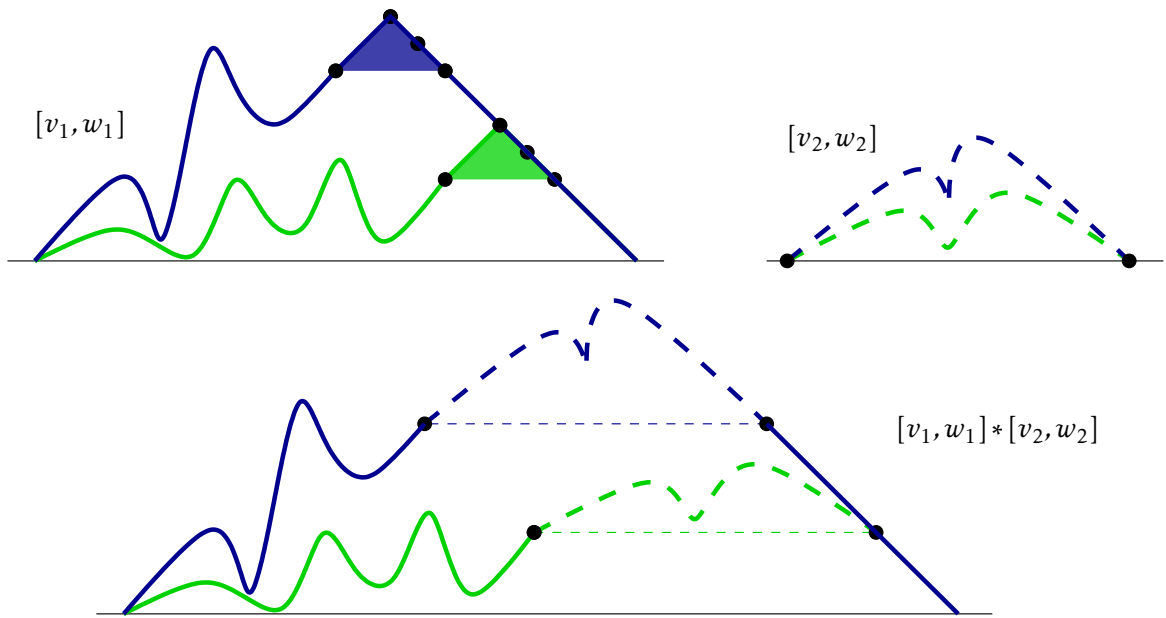}
        \caption{The monoid structure on greedy intervals, when $m=2$.}
        \label{fig:product-intervals}
      \end{figure}

This lemma allows us to define a monoid structure on the set of all intervals of $\eD_m= \cup_{n\ge 1} \eD_{m,n}$ as follows. We  intentionally use for this monoid the same symbol $*$ as for the monoid on Dyck paths.

Let $[v_1,w_1]$ and $[v_2,w_2]$ be intervals. By Proposition~\ref{covering_and_product}, we have $v_1*v_2\le  w_1*w_2$.  We define the product
$[v_1,w_1]*[v_2,w_2]$ to be the interval $[v_1*v_2, w_1*w_2]$ (see Figure~\ref{fig:product-intervals}).
For instance, for $m=1$ we have $[10{\color{red}{10}}, 1{\color{red}{10}}0] * [110010, 111000]= [10110010, 11110000]$.
Since~$*$ is associative on Dyck paths, it is also associative on intervals. The unit is the interval $[10^m,10^m]$. We denote by $\e{I}_m$ the monoid of intervals of positive size. 

\begin{prop}\label{prop:Im}
  The monoid $\e{I}_m$ is free over the   generators $[v,w]$ where  $v$ is a generator of the monoid $\eD_m$, listed in Proposition~\ref{prop:monoid_words}, and $w$ any Dyck path larger than or equal to $v$.
\end{prop} 

\begin{proof}
  Let $[v,w]$ be an interval that is  not the unit of
  $\e{I}_m$. By unique factorisation in the monoid $\eD_m$, one can write
  $v = v_1 * v_2 * \dots * v_k$ for some generators $v_i$. By
  Proposition~\ref{covering_and_product}, 
  we have $w = w_1 * w_2 * \dots * w_k$ where $v_i \le w_i$ for each $i$.
  Consequently, $[v,w]= [v_1, w_1] * \cdots * [v_k, w_k]$. Hence the elements $[v,w]$, for $v$ a generator of $\eD_m$ and $w\ge v$, indeed generate the monoid~$\e{I}_m$.

For instance, for $m=1$ and the above interval $[v,w]= [10110010, 11110000]$, the factorisation of $v$ is $v_1*v_2= 1010 * 110010= D(\vide, 10) * D(10, 10)$. Hence we write $w$ as the product $w_1 *w_2=1100* 111000$ of a word of size $2$ and a word of size $3$, and obtain $[v,w]=[v_1, w_1] * [v_2, w_2]$.
  
  Let us now pick any factorisation of $[v,w]$ as a product of
  generators $[v_i,w_i]$. That is,  $v = v_1 * \cdots * v_k$ and  $w = w_1 *  \cdots * w_k$ with $v_i\le w_i$ for all $i$. Since the $v_i$'s are generators of the free monoid $\eD_m$, the sequence $(v_1, \ldots, v_k)$ is uniquely determined by $v$. Moreover, given that~$v_i$ and $w_i$ have the same size for every $i$, the sequence $(w_1, \ldots, w_k)$ is uniquely determined by~$(v,w)$. (Indeed, given a product  $u*u'$, one can recover $u$ and $u'$ from  $u*u'$ as soon as  we  know the size of $u'$: one computes the unique factorization $u_1 * \cdots * u_\ell$ of $u*u'$ in the generators of Proposition~\ref{prop:monoid_words}, and then $u'$ is the only factor $u_j * \cdots * u_\ell$, with $1\le j\le \ell$, that has the same size as $u'$.)
    This proves uniqueness of the factorisation for the interval $[v,w]$, and concludes the proof.
\end{proof}

\section{Recursive decomposition of greedy intervals}
\label{sec:dec}

The aim of this section is to establish a functional equation that characterises the \gf\ $\sI$ of greedy $m$-Tamari intervals. In the series $\sI$, the interval $[v,w]$ is weighted by $t^n x^d$, where $n\ge 1$ is the common size of $v$ and $w$ (called the size of the interval) and $d=d(w)$ is the length of the final descent of $w$. The interval $[\vide, \vide]$ is thus not counted in $\sI$. The functional equation for $\sI$ involves the divided difference operator $\Delta$ defined by:
  \beq\label{Delta-def}
    \Delta \s F := \frac{\s F - \s F(1)}{x-1},
  \eeq
where, for a \fps\ $\s F$ in $t$ with coefficients in $\qs[x]$ (the ring of polynomials in~$x$), we denote by $\s F(1)$ the specialisation of $\s F$ at $x=1$. 

\begin{prop}\label{prop:func-eq}
  The bivariate \gf\ $\sI$ of greedy $m$-Tamari intervals is the unique solution of the following equation:
\[ 
  x^2\sI =t(x+x^2\sI\Delta)^{(m+2)}(1),
  \] 
  where the first $x$ in the operator $(x+x^2\sI\Delta)$ stands for the multiplication by $x$, and the exponent $(m+2)$ means that the operator is applied $m+2$ times.
\end{prop}
We will  denote $\hI= x^2\sI$, so that the above equation reads
\[ 
\hI=t(x+\hI \Delta)^{(m+2)}(1).
\] 

\noindent
{\bf Example.} When $m=1$, the equation reads
\begin{align*}
  \hI& =t(x+\hI \Delta)^{(3)}(1)\\
     &=t(x+\hI \Delta)^{(2)}(x) \\
  &=  t(x+\hI \Delta)(x^2+\hI)\\
  & = t x(x^2+\hI) + t\hI \left( x+1 + \frac{\hI-\hI(1)}{x-1}\right).
\end{align*}

In order to prove the above
proposition, we first introduce several families of greedy intervals, and provide recursive decompositions for them.
These decompositions translate into a system of functional equations for the corresponding   bivariate \gfs.  This system finally results in the above proposition. Solving the above functional equation (and in fact, the entire system) will be  the topic of Section~\ref{sec:sol}.
 
\subsection{Some families of intervals}
Recall that we only consider intervals of positive size, ignoring the interval $[\vide, \vide]$. The set of all such intervals has been so far denoted by $\e{I}_m$ so far, but we will now drop the index $m$ and simply write $\e I$. We  introduce the following collection of subsets of $\e{I}$.

\begin{definition}\label{def:Ji}
  For $i \in \llbracket 0, m+1\rrbracket$,  let
  $\e{J}_i$ denote the  set of  intervals $[v,w]$ such that the minimum~$v$ is   of the form  $D(v_1,v_2,\dots,v_i,\vide,\vide,\dots,\vide)$, for some $v_1, \ldots,  v_i$ in $\e D_m \cup \{\vide\}$.
\end{definition}
Observe  that $\e{J}_0$ only contains   the unit interval $[10^{m},10^{m}]$,
  and that $\e{J}_{m+1}=\e{I}$. Denoting by $\sJ_i$ the bivariate \gf\ of $\e{J}_i$, we will establish the following system:
  \[
    \left\{
      \begin{array}{ll}
        \sJ_0&= x^mt,\\
        \sJ_i  &=\displaystyle\sJ_{i-1} +  \sI \frac{x\sJ_{i-1} -x^{m+1-i} \sJ_{i-1}(1)}{x-1} \qquad \text{for } i>0.\\
           \end{array}
    \right.
  \]
As explained in Section~\ref{sec:single-eq}, Proposition~\ref{prop:func-eq} easily follows.

As an intermediate step in the decomposition of the intervals of $\e{J}_i$, it will be convenient to introduce the following subset of $\e{J}_i$. 

\begin{definition}
  For $i \in \llbracket 1  , m+1\rrbracket$,
  let  $\e{K}_i$ denote the subset of $\e{J}_i$ consisting of intervals $[v,w]$ such that $v$ is of the form
   $D(v_1,v_2,\dots,v_{i-1},10^m,\vide,\vide,\dots,\vide)$.
\end{definition}

The associated bivariate generating function is denoted by $\s{K}_i$.

\subsection{Description of $\e{J}_i$}
\label{sec:Ji}

 If $i=0$, then $\e{J}_i=\e{J}_0$ is reduced to $[10^m, 10^m]$, so that $×\sJ_0=x^mt$.  We now assume $i\ge 1$.

 Let $[v,w]$ be an interval of $\e{J}_i$, and write $v=D(v_1, \ldots, v_i, \vide, \ldots, \vide)$. If $v_i=\vide$, then $[v,w]$ is any interval of $\e{J}_{i-1}$. Let us now assume that $v_i \not = \vide$.   Let us write $[v,w]= [v', w'] *[v'', w'']$, where $[v',w']$ is a generator of the free monoid $\e I_m\equiv\e{I}$ (see Proposition~\ref{prop:Im}).
 Recall that $v'$ is also the first factor in the factorisation of $v$ in the free monoid~$\eD_m$, so that, by Lemma~\ref{lem:first_factor}, $v'=D(v_1,\dots,v_{i-1}, 10^m,\vide, \ldots, \vide)$ (and $v''=v_i$). This means that $[v',w']$ belongs to $\e{K}_i$. Let us denote $\phi ([v,w]) = ( [v',w'], [v'', w''])$ (see Figure~\ref{fig:dec-Bj}).

For instance, for $m=1$ and $[v,w]=[110100101100,1101 111000    00]$, we have $v=D(1010, \allowbreak 101100)$ so that $[v,w] \in \e{J}_2$. The path $v$ factors as  $v= v'*v''$ with $v'= 11010010= D(1010,10)$ and $v''= 101100$. The interval $[v,w]$ decomposes as $ [v', w'] *[v'', w'']$ where $w'=11011000$ and $w''= 111000$, and indeed $ [v', w'] \in \e{K}_2$.

 Conversely, take $[v',w']$ in $\e K_i$, and $[v'',w'']$ in $\e{I}$. Observe that $v'$ is a generator of $\eD_m$, and $[v',w']$ a generator of $\e{I}_m$. Form the interval $[v,w]:= [v', w'] *[v'', w'']$. Then $[v',w']$ is the first generator in the factorisation of $[v,w]$ and $\phi ([v,w]) = ( [v',w'], [v'', w''])$.

 We have thus described a bijection $\phi$ that maps $\e{J}_i \setminus \e{J}_{i-1}$ onto $\e K_i \times \e{I}$. Moreover, if $\phi ([v,w]) = ( [v',w'], [v'', w''])$, then $|w|=|w'|+|w''|-1$ and the final descent of $w$ has length $d(w)=d(w')+d(w'')-m$.
This gives the following identity:
\begin{equation}
  \label{B_de_CA}
  {\sJ_i}= \sJ_{i-1}   + \frac{\s{K}_i}{x^m t} {\sI}.
\end{equation} 

\begin{figure}[htp]
  \centering
  	\scalebox{0.9}{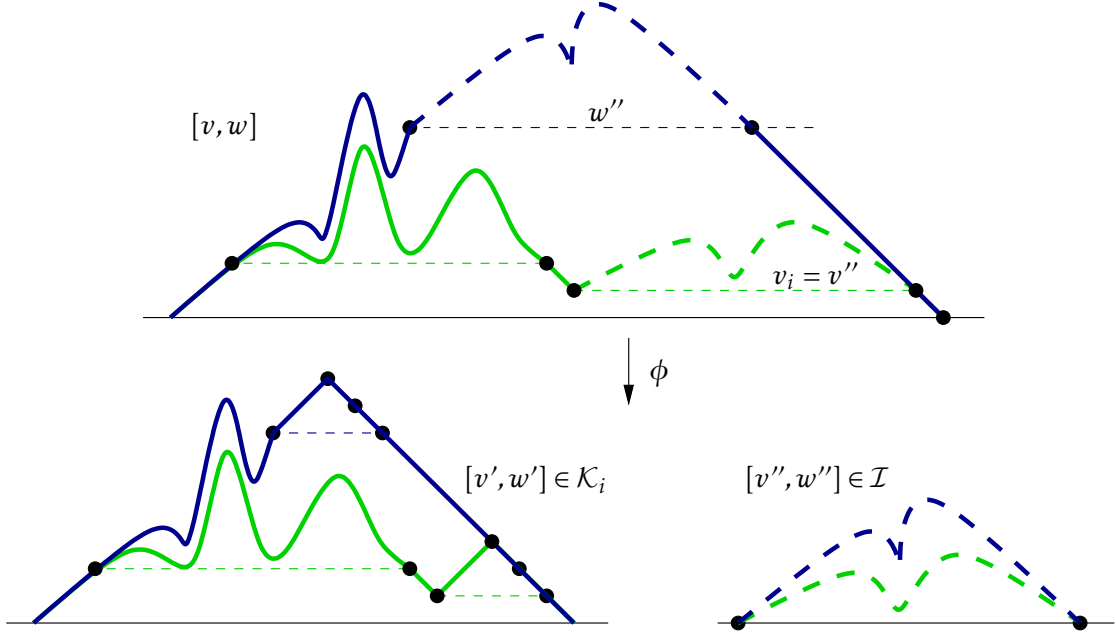}
        \caption{Decomposition of an interval $[v,w]$ of $\e{J}_i\setminus \e{J}_{i-1}$, with $m=2$ and $i=2$.}
        \label{fig:dec-Bj} 
\end{figure}
   
\subsection{Description of $\e{K}_i$}

Let us begin with a simple observation.
  \begin{lemma}\label{peak-delete}
Consider a cover relation $v \cov w$, where $v$ and $w$ are non-empty. Let $v'$ (resp.~$w'$) be obtained by deleting the last peak of $v$ (resp.~$w$). Then either $v'=w'$ (if the cover relation takes place in the last valley of $v$) or otherwise $v' \cov w'$. Consequently, if $v \le w$ and $v'$ and $w'$ are obtained as above, then $v'\le w'$.
  \end{lemma}

  Now let $i \in \llbracket 1, m+1\rrbracket$, and consider an interval $[v,w]$ in $\e K_i$. By definition, this means that $v=D(v_1, \ldots, v_{i-1}, 10^m, \vide, \ldots, \vide)$. Let us define $v'$ and $w'$ by deleting the last peak of $v$ and $w$, respectively. By Lemma~\ref{peak-delete}, we obtain an interval $[v',w']$. Moreover, $v'=D(v_1, \ldots, v_{i-1},  \vide, \ldots, \vide)$, so that $[v',w']$ is  in $\e{J}_{i-1}$. To recover $v$ from $v'$, it suffices to insert the factor $10^m$ in the final descent of $v'$, at height $m+1-i$ (by this, we mean that the final up step of $v$ starts at  height $m+1-i$). Analogously, we  recover $w$ from $w'$ by inserting a peak  in the final descent of $w'$. Note that its insertion  height $h$ is at least $m+1-i$ (so that $w$ is above $v$) and at most $d(w')$. Let us denote $\psi([v,w]):=( [v',w'], h)$.

For instance, take $m=1$ and $[v,w] = [11010010,11011000]$. We have $v= D(1010,10)$ so that $[v,w] \in \e{K}_2$. Removing the final peaks of $v$ and $w$ gives $v'= 110100= w'$. In particular, $[v',w']$ is trivially an interval. We recover $v$ from $v'$ by inserting a peak $10$ in~$v'$ at height $0$, and we recover $w$ from $w'$ by inserting $10$ at height $2$. Hence $\psi([v,w]):=( [110100,110100], 2)$.

The converse construction is illustrated in Figure~\ref{fig:dec-Cj}.  Take $[v',w']$ in $\e{J}_{i-1}$, and write $v'=D(v_1, \ldots, v_{i-1},  \vide, \ldots, \vide)$. 
Form the path $v:=D(v_1, \ldots, v_{i-1},  10^m,\vide, \ldots, \vide)$. Choose $h \in \llbracket m+1-i, d(w')\rrbracket$, and let $w$ be obtained by inserting $10^m$ in the final descent of $w'$, at height $h$. Let us also introduce an intermediate path $w_0$, obtained by inserting a peak at height $m+1-i$ in the final descent of $w'$. Let us now prove that $v\le w_0 \le w$. Let us choose  a sequence of cover relations from $v'$ to $w'$. Since they take place at valleys of height $>m+1-i$ (there is no lower valley in~$v'$), performing cover relations at the same places in $v$ gives a sequence of cover relations from $v$ to $w_0$; see Figure~\ref{fig:dec-Cj}(2). Now starting from $w_0$, we perform  a sequence of cover relations taking place systematically in the last  valley, until the final peak starts at height $h\le d(w')$. These cover relations only move the last peak up; see Figure~\ref{fig:dec-Cj}(3). 
Thus $v\le w$, and the interval $[v,w]$ belongs to $\e K_i$ since $v=D(v_1, \ldots, v_{i-1},  10^m,\vide, \ldots, \vide)$. Moreover, $\psi([v,w])=( [v',w'], h)$.
     
\begin{figure}[htp]
  \centering
  	\scalebox{0.9}{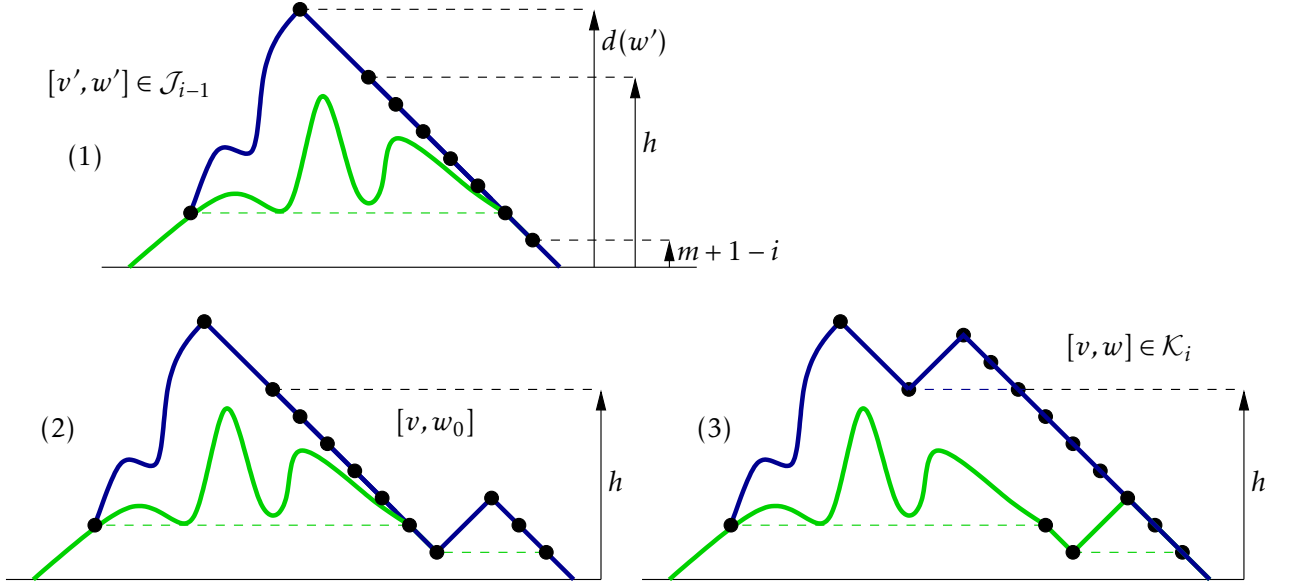}
            \caption{From an interval $[v',w']$ in $\e{J}_{i-1}$ and a height $h$ in $\llbracket m+1-i, d(w')\rrbracket$ to an interval $[v,w]$ in $\e K_i$. Here, $m=2$, $i=2$ and $h=7$.}
          \label{fig:dec-Cj} 
\end{figure}
 
We have thus described a bijection $\psi$ between $\e K_i$ and the set of pairs $([v',w'],h)$ such that $[v',w']\in \e{J}_{i-1}$ and $ h \in \llbracket m+1-i, d(w')\rrbracket$.
 Moreover, if $\psi([v,w])=( [v',w'], h)$, then $|v|=1+|v'|$ and $d(w)=m+h$.
In terms of \gfs, this gives
\begin{equation}
  \label{C_de_B} 
  \s{K}_i = \sum_{[v',w'] \in \e{J}_{i-1}} t^{1+|v'|}\sum_{h=m+1-i}^{d(w')} x^{m+h}
    = x^{m} t  \frac{x\sJ_{i-1}-x^{m+1-i}\sJ_{i-1}(1)}{x-1}.
\end{equation}

\subsection{Proof of Proposition~\ref{prop:func-eq}}
\label{sec:single-eq}

We now combine the above functional equations into a single equation defining $\sI$. First, we use~\eqref{C_de_B} to rewrite the expression~\eqref{B_de_CA} of $\sJ_i$. For $1\le i\le m+ 1$, we thus obtain
\[
  \sJ_i  =\sJ_{i-1} +  \sI \frac{x\sJ_{i-1} -x^{m+1-i} \sJ_{i-1}(1)}{x-1},
\]
or equivalently,
\begin{equation}
  \label{next_J}
  \frac{\sJ_i}{x^{m-i-1}}= \left( x+x^2 \sI \Delta\right)\left( \frac{\sJ_{i-1}}{x^{m-i}}\right),
\end{equation}
where $\Delta$ is the divided difference operator~\eqref{Delta-def}. Using $\sJ_0=x^m t$, this can be solved by induction on $i$ as
\[ 
  \frac{\sJ_i}{x^{m-i-1}}= \left( x+x^2 \sI \Delta\right)^{(i)}(xt)
  = t\left( x+x^2 \sI \Delta\right)^{(i+1)}(1).
\] 
Given that $\s I=\s J_{m+1}$, this gives the equation  of Proposition~\ref{prop:func-eq}.
\qed

\section{Solution of the functional equation}
\label{sec:sol}

In this section, we solve the functional equation defining the series $\sI$  in Proposition~\ref{prop:func-eq}. For any $m$, this is an equation in one ``catalytic'' variable $x$, in the terminology of~\cite{mbm-jehanne}. In particular, it follows from the latter reference that $\sI$ is an \emm algebraic, series. That is, it satisfies a non-trivial polynomial equation with coefficients in $\qs[x,t]$.
Moreover, one can also use the tools developed in~\cite{mbm-jehanne}, and very recently in~\cite{BNS}, to construct an explicit algebraic equation satisfied by $\sI$, for small values of $m$. However, the difficulty here is to determine $\sI$ for an \emm arbitrary, value of $m$.

The functional equation satisfied by $\sI$  is reminiscent of the equation satisfied by a bivariate \gf\ $\s T$ of ordinary Tamari intervals in $\eD_{m,n}$, $n\ge 0$, which was established in~\cite[Prop.~8]{mbm-fusy-preville}, and reads:
\beq\label{eq-ord-contacts}
  \s T =x + xt \left( \s T \Delta \right)^{(m+1)}(x).
  \eeq
  It is also reminiscent of an equation derived in~\cite[Thm.~4.1]{fang-these} for planar $(m+1)$-constellations, which reads:
  \beq\label{eq:const}
    \s C = 1 +xt (\s C + \Delta)^{(m+1)}(1).
  \eeq
In the latter equation, $t$ records the number of polygons and $x$ the degree of the white root face, divided by $(m+1)$. For $m=2$, the solution to this equation starts
  \[
       \sC=1+tx + t^2\left( 3x^2 + 3 x\right)+ \LandauO(t^3),
  \]
(see Figure~\ref{fig:const})  while the series counting greedy $m$-Tamari intervals reads
  \[
  \sI=  tx^2+ t^2\left( 3x^2+2x+1\right)+\LandauO(t^3) .
  \]
  Hence, even though we will prove that $1+\sI(1)=\sC(1)$ (for any $m$), the catalytic parameters do not match.
  We believe that the catalytic parameter used for constellations corresponds to the length of the first ascent {of the maximal element} in greedy intervals, and refer to Section~\ref{sec:conj} for a much refined conjecture.

  Our approach to solve the equation of  Proposition~\ref{prop:func-eq} is similar to the one used in~\cite{mbm-fusy-preville}: by examination of the solution for small values of $m$, we guess a general parametric form of the solution, valid for any $m$, and then check that this guess satisfies the functional equation. More precisely, we guess the value of all series $\sJ_i$, for $1\le i \le m+1$, and prove that these values satisfy the system~\eqref{next_J}.

\medskip
We  introduce a rational parametrization of $t$ and $x$ by two formal power series in $t$ denoted~$\sZ$ and $\sU$. 
The series $\sZ$  has integer coefficients, while $\sU$  has coefficients in $\qs[x]$. The series $\sZ$ is the unique \fps\ in $t$ with constant term $0$ such that
\[ 
  t=  \sZ  {(1-\mp \sZ )^{m}},
  \] 
  where we denote $\mp:=m+1$ to avoid having too many parentheses 
  around, and $\sU$ is the unique \fps\ in $t$ such that
\beq\label{x-param}
  x= \frac \sU  {1-\mp \sZ } \left(1- \sZ \frac{\sU ^{m+1}-1}{\sU -1}\right).
  \eeq
 We have 
\[
\sZ=t + \LandauO(t^2) \qquad  \text{and} \qquad  \sU =x+ xt \left( \frac{x^{m+1}-1}{x-1} -m^+\right) + \LandauO(t^2).
\]
Note also that $\sU(1)=1$. We have found a  rational expression of $\s I$, and in fact of all series $\s J_i$ with $0\le i \le {m+1}$, in terms of the series $\s Z$ and $\sU$.

\begin{theorem}\label{thm:series}
  The bivariate \gf\  $\sI$ of  intervals in the greedy $m$-Tamari posets, counted by the size and the final descent of the maximal element, is given by:
\beq\label{bF-param}
x^2  \sI =
\frac{\sZ \sU ^{m+2}}{(1-\mp \sZ )^2} \left( 1- \sZ \sum_{e=0}^m \sU ^e (m+1-e)
\right)
=\frac{\sZ \sU ^{m+2}}{1-\mp \sZ }\cdot\frac{x-1}{\sU -1 }
.
\eeq
In particular, the size \gf\ of greedy $m$-Tamari intervals is
\[
  \sI(1)= \frac{\sZ}{(1-\mp \sZ )^2} \left( 1- \binom{m+2}{2}\,\sZ \right).
\]

More generally,  for $0 \leq i \leq m+1$, the bivariate series $\sJ_i$ that counts the intervals of the set $\e J_i$ (see Definition~\ref{def:Ji})  is given by:
 \beq\label{Ji-expr}
   \frac{\sJ_i}{x^{m-i-1}}= \s Z (1 - m^+ \s Z)^{m-i-1}\s \sH_i(\sZ; \sU),
   \eeq
  where     $\sH_i(z;u)\equiv \sH_i(u)$ is a  polynomial in $z$ and $u$ given by:
   \beq\label{H-m+1}
  \s H_{m+1}(u)= u^{m+2} \left( 1- z \sum_{e=0}^m u ^e (m+1-e)  \right),
  \eeq
and for $0\le i \le m$,
   \beq\label{Hi-def}
   \sH_i(u)=
   u^{i+1}   \left( 1+  \sum_{k=1}^{i+1} (-z)^k  \binom{i+1}{k} \sum_{e=0}^{m-i} \binom{e+k-1}{e} u^e 
         + \sum_{k=1}^{i}  (-z)^k  \binom{m+k-i-1}{k-1}   \sum_{e=0}^{i-k} \binom{e+k}{e} u^{m+1-k-e}\right).
     \eeq
\end{theorem}

\noindent{\bf Remark.}
We can apply  the Lagrange inversion formula to the series $\sZ$ and the above expression of~$\sI(1)$, and this gives the expression of Theorem~\ref{thm:numbers} for the number of greedy Tamari intervals in $\eD_{m,n}$. The coefficients of the bivariate series $\sI$ do not seem to factor nicely.

\begin{proof}
  Let us denote $\hJ_i := \s J_i / x^{m-i-1}$ for $i=0,\dots,m+1$. Recall that we have  also denoted $\hI = x^2 \sI = \hJ_{m+1}$.
  The functional equations~\eqref{next_J} become, for $1\le i \le m+1$, 
  \begin{equation}
    \label{next_hJ}
    \hJ_{i}= \left( x+ \hI \Delta\right) \hJ_{i-1}.
  \end{equation}
  Together with the initial condition $\hJ_0 = xt$, and the fact that the series  $\hJ_i$ have no constant term in $t$ (as they count intervals of positive size), these equations define each  $\hJ_i$ uniquely as a formal power series in $t$. Indeed, the coefficient of
  $t^n$ in~$\hJ_i$ is a polynomial in $x$ that can be computed by a double
  induction on $n$ and $i$.  It is clear on~\eqref{bF-param} and~\eqref{Ji-expr} that the claimed values of $\hI$ and $\hJ_{i}$ have no constant term in $t$, and moreover one readily checks that the right-hand side of~\eqref{Ji-expr} reduces to $xt$ when $i=0$, as expected. Thus it suffices to prove that the series
  $\hI$ and $\hJ_i$ of Theorem~\ref{thm:series}
  satisfy all the equations~\eqref{next_hJ}. We will see that these equations simplify nicely once rewritten in terms of  $\sZ$ and $\sU$.

\smallskip

Let $\sH(u)$ be an arbitrary polynomial in $u$, and recall that the series $\sU$ defined by~\eqref{x-param} equals~$1$ when $x=1$. Hence, if $\hI$ is given by the right-hand side of~\eqref{bF-param}, we have
\[ 
  \left( x + \hI \Delta \right)\left( \sH(\sU)\right)= \frac \sU{1-\mp \sZ}
\left(  \sH(\sU) +\sZ\,  \frac{\s H(\sU) - \sU^{m+1} \s H(1)}{\sU-1}\right).
\]
Hence the series $\hJ_i=\sJ_i/x^{m-i-1}$ given by~\eqref{Ji-expr} satisfy the system~\eqref{next_hJ} if and only if, for $1\le i \le m+1$,
\[
    \s H_i (\sU) =  \sU
\left(  \sH_{i-1}(\sU) +\sZ\,  \frac{\s H_{i-1}(\sU) - \sU^{m+1} \s H_{i-1}(1)}{\sU-1}\right).
\]
  This system holds with the series $\sH_i(z;u)$ evaluated at $(\sZ, \sU)$ if it holds with indeterminates $(z,u)$, that is, if for $1\le i \le m+1$,
  \beq\label{syst-H}
  \s H_i (u)= u \sH_{i-1}(u) + z u\, \frac {\s H_{i-1}(u) - u^{m+1} \s H_{i-1}(1)}{u-1}.
\eeq
In short,
\[ 
\s H_i =\left( u+z\nablam\right)\left(\sH_{i-1}(u)\right),
\]
where the  operator  $\nablam$ is defined by
\begin{equation} 
  \label{def_nabla}
  \nablam \s H = u \frac{\s H - u^{m+1} \s H(1)}{u-1}.
\end{equation}
This is now a  polynomial identity, which we prove in Appendix~\ref{sec:app} using basic binomial identities.
\end{proof}

 \section{A new solution for ordinary $m$-Tamari intervals}\label{sec:sol-ord}

 In this section, we show how to adapt the decomposition of greedy $m$-Tamari intervals used in Section~\ref{sec:dec} to count \emm ordinary, $m$-Tamari intervals. We thus obtain a second proof of the result of~\cite{mbm-fusy-preville}, giving the number of such intervals in $\mathcal D_{m,n}$ in the form~\eqref{ordinary}. Moreover, we refine this result by recording the length of the final descent in the upper path of the interval (Theorem~\ref{thm:series-ord}), while the result of~\cite{mbm-fusy-preville} was recording instead the number of contacts of the lower path (and the length of the first ascent of the upper path).

 The key difference with what has been done in Section~\ref{sec:dec}
 for greedy intervals is that we now only consider factorisations of $m$-Dyck paths $v$ of the form $v_1 * v_2$ such that $v_2$ is \emm prime,, that is, has no contact (recall that the endpoints do not count as contacts). The proofs are very close to the greedy case, and we will be a bit more sketchy in this section.  We will use the following counterpart of Proposition~\ref{prop:monoid_words} and Lemma~\ref{lem:first_factor}.

 \begin{lemma}\label{lem:fact-ord}
   Let $v=D(v_1, \ldots, v_{j-1}, v_j, \vide, \ldots, \vide)\in \mathcal D_{m,n}$, with $v_j \not = \vide$. Then $v$ can be written in a unique way as $v' * v''$, where $v'$ is of the form $D(v'_1, \ldots, v'_{j-1}, v'_j 10^m, \vide, \ldots, \vide)$ and $v''$ is prime.
 \end{lemma}
 \begin{proof}
Write $v_j= v'_j 10^m * v''$, where $v''$ is prime. Clearly, there is a unique way of doing this (the path $v''$ is the unique prime Dyck path that is a suffix of $v_j$).    Then the only factorisation of $v$ that satisfies the conditions of the lemma is obtained with $v'_i=v_i$ for $1\le i <j$.
\end{proof}
For instance, when $m=2$, the path $v=110011000000=D(100110000,\vide, \vide)$ factors as $D(100100, \vide, \vide) * 110000$, and $v'':=110000$ is prime.

\medskip

We now have the following counterpart of Proposition~\ref{covering_and_product}.
\begin{prop}
  \label{covering_and_product-ord}
  Let $v = v_1 * v_2$ be a non-empty Dyck path, and assume that $v_2$ is prime. Let $v \cov w$ be an ordinary Tamari cover relation.  Then either $w = w_1 * v_2$ where $w_1$ covers $v_1$, or $w=v_1 * w_2$ where $w_2$ covers $v_2$.
 
  Conversely, every cover relation $v_1 \cov w_1$ gives a
  cover relation $v_1 * v_2 \cov w_1 * v_2$ and every cover relation
  $v_2 \cov w_2$ gives a cover relation $v_1 * v_2 \cov v_1 * w_2$.

 Consequently, for any non-empty Dyck path $v = v_1 * v_2$, with $v_2$ prime, the upper ideal $\{w: v \leq w\}$ is $\{w_1 *w_2: v_1\le w_1 \text{ and } v_2\le w_2\}$.
\end{prop}

This proposition allows us to define the product $[v_1, w_1]* [v_2,w_2]$, again as $[v_1*v_2, w_1 *w_2]$, but now under the assumption that $v_2$ is prime. Note that this implies that $w_2$ is prime too.
 
 \subsection{Recursive description of ordinary intervals}
As before, we only consider (ordinary) intervals of positive size. Let us denote by $\ebI\equiv \ebI_m$ the set of such intervals. For $1\le i \le m+1$, let $\ebJ_i$ be the set of intervals $[v,w]$ such that $v=D(v_1, \ldots, v_{i-1}, v_i, \vide, \ldots, \vide)$. Finally, let $\ebK_i$ be the subset of $\ebJ_i$ consisting of intervals $[v,w]$ such that $v=D(v_1, \ldots, v_{i-1} ,v_i 10^m, \vide, \ldots, \vide)$. The associated \gfs\ are denoted by $\sbI$, $\sbJ_i$ and $\sbK_i$, respectively. Observe that $\ebJ_0$ only contains the unit interval $[10^m, 10^m]$, while $\ebJ_{m+1}$ coincides with $\ebI$.
 
\begin{prop}\label{prop:func-eq-ord}
  The series $\sbJ_i$, for $0 \le i \le m+1$, are given by  $\sbJ_0=x^mt$ and the equations:
\[
  \sbJ _i= \sbJ_{i-1}+\sbJ_m \frac{x \sbJ_i - x^{m+1-i} \sbJ_i(1)}{x-1}
  , \qquad \text{for } 1\le i \le m+1.
 \]
 Recall that $\sbJ_{m+1}$ coincides with the bivariate \gf\ $\sbI$ of ordinary $m$-Tamari intervals.
\end{prop}

\begin{proof}
  The proof follows the same steps as in Section~\ref{sec:dec}. First, if $[v,w]\in \ebJ_{i} \setminus \ebJ_{i-1}$, we write $v=v'*v''$ as in Lemma~\ref{lem:fact-ord}. By Proposition~\ref{covering_and_product-ord}, we have $[v,w]=[v',w'] *[v'',w'']$. The map that sends $[v,w]$ to the pair $([v',w'] ,[v'',w''])\in \ebK_i \times \ebJ_m$ is easily seen to be bijective. This is the counterpart of Section~\ref{sec:Ji}, and gives
  \[
    \sbJ _i= \sbJ_{i-1}+ \frac{\sbK_i}{x^m t} \sbJ_m.
  \]
For instance, take $m=1$ and consider the interval $[v,w]=[110100101100,1101 111000    00]$, already considered in the greedy case (Section~\ref{sec:Ji}). The factorisation that we use is now $[v',w']*[v'',w'']$ where $v'= D(1010,1010)$, $v''=1100$, $w'=1101110000$ and $w''= 1100=v''$.
  
  Let us now decompose intervals of $\ebK_i$. First, we check that Lemma~\ref{peak-delete} holds verbatim for ordinary $m$-Tamari intervals.  Now take $[v,w]\in \ebK_i$, with $v= D(v_1, \ldots, v_i10^m, \vide, \ldots, \vide)$.  Define~$v'$ (resp. $w'$) by deleting the last peak of $v$ (resp.~$w$). In particular, $v'= D(v_1, \ldots, v_i, \vide, \ldots, \vide)$, so that the interval $[v',w']$ belongs to $\ebJ_i$. Let $h$ be the height of the starting point of the last 
  up step in $w$. Again, we have the natural bounds $h\in \llbracket m+1-i, d(w')\rrbracket$. The map $\psi$ defined by $\psi([v,w])= ([v',w'], h)$ is easily seen to be a bijection from $\ebK_i$ to $\{ ([v',w'],h) \in \ebJ_i\times \ns : m+1-i\le h \le d(w")\}$, and the equation
 \[
   \sbK_i= x^m t \frac{x \sbJ_i - x^{m+1-i} \sbJ_i(1)}{x-1}
 \]
 is obtained as the counterpart of~\eqref{C_de_B}.  We now combine this equation with the previous one to complete the proof of the proposition.  
\end{proof}

\noindent{\bf Remark.} When $m=1$, it is known that the length of the last descent of the upper path~$w$ in ordinary intervals $[v,w]$ is distributed as the number of contacts in the lower path~$v$, plus one (in fact, the joint distribution is symmetric, see~\cite[Sec.~4]{mbm-fusy-preville}). In terms of \gfs, this means that $\sbI= \s T /x-1$, where $\s T$ is the solution of~\eqref{eq-ord-contacts} when $m=1$. We could thus expect that the functional equation obtained for $\sbI=\sbJ_2$ by elimination of $\sbJ_1$ in the system of Proposition~\ref{prop:func-eq-ord} coincides with the one derived from~\eqref{eq-ord-contacts}. This is however not the case: the former equation involves the series $\sbJ(1)$ and $\sbJ'(1)$, while the latter only involves~$\sbJ(1)$. Mixing both equations provides a relation between these two series.
 
\subsection{Solution}

As before, we  introduce a rational parametrization of $t$ and $x$ by two formal power series in $t$ denoted~$\sbZ$ and $\sbU$.
The series $\sbZ$  has integer coefficients, while $\sbU$  has coefficients in $\qs[x]$. The series $\sbZ$ is the unique \fps\ in $t$ with constant term $0$ such that
\[ 
  t=  \sbZ  \left(1-\sbZ\right )^{m^2+2m},
  \] 
  and $\sbU$ is the unique \fps\ in $t$ such that
\[ 
  x= \frac \sbU  {(1-\sbZ)^{m+2} } \left(1- \sbZ \frac{\sbU ^{m+1}-1}{\sbU -1}\right).
  \] 
 We have 
\[
\sbZ=t + \LandauO(t^2) \qquad  \text{and} \qquad  \sbU =x+ xt \left( \frac{x^{m+1}-1}{x-1} -(m+2)\right) + \LandauO(t^2).
\]
Note that $\sbU(1)=1-\sbZ$.

\begin{theorem}\label{thm:series-ord}
  The bivariate \gf\  $\sbI$ of  intervals in the ordinary $m$-Tamari lattices, counted by the size and the final descent of the maximal element, is given by:
\[ 
x^2  \sbI =
\frac{\sbZ\, \sbU ^{m+2}}{(1-\sbZ )^{2m+4}} \left( 1- \sbZ \sum_{e=0}^m \sbU ^e (m+1-e)\right).
\] 
In particular, the size \gf\ of ordinary $m$-Tamari intervals satisfies
\[
1+  \sbI(1)= \frac{1-(m+1)\sbZ}{(1-\sbZ)^{m+2}},
\]
as already established in~\cite{mbm-fusy-preville}.
We then recover~\eqref{ordinary} using the Lagrange inversion formula.

Moreover, we have
\beq\label{Jm-ord}
  \sbJ_m= \frac{x-1}x  \cdot \frac{\sbZ\, \sbU^{m+1}}{\sbU-1+\sbZ},
\eeq
More generally,  for $0 \leq i \leq m+1$, the bivariate series $\sbJ_i$  is given by:
 \beq\label{Ji-expr-ord}
    \frac{\sbJ_i}{x^{m-i-1}}= \sbZ \left(1 -  \sbZ\right)^{(m+2)(m-i-1)}\sH_i\left(\sbZ; \sbU\right),
 \eeq
where the polynomials    $\sH_i(z;u)\equiv \sH_i(u)$ are defined in~\eqref{H-m+1}-\eqref{Hi-def}.
\end{theorem}
\noindent{\bf Remarks}\\
{\bf 1.} One readily checks that the  expression~\eqref{Jm-ord}  of $\sbJ_m$ equals the right-hand side of~\eqref{Ji-expr-ord} when~$i=m$.
\\
{\bf 2.}
It is striking
that the expressions of $\sbJ_i$ (for ordinary intervals) and $\sJ_i$ (for greedy intervals, see~\eqref{Ji-expr}) are so close. In fact, our proof of the above theorem uses the solution of the greedy case.

\begin{proof}
  The system of Proposition~\ref{prop:func-eq-ord}, the initial condition $\sbJ_0=tx^m$, plus the fact that the series~$\sbJ_i$ have no constant term in $t$, characterize these series as formal power series in $t$. The claimed values of the $\sbJ_i$'s have no constant term, and it is easy to see that the initial condition holds as well. It thus suffices to prove that they satisfy the system. 

  We argue as in the proof of Theorem~\ref{thm:series}, but this time we have $\sbU=1-\sbZ$ when $x=1$. The system that the series $\sH_i$ must now satisfy (the counterpart of~\eqref{syst-H}) reads:
  \beq\label{Hi-new-eq}
\sH_i(u)= \tilde x(1-z)^{m+2} \sH_{i-1}(u)+ z u^{m+1} \frac{\sH_i(u) -\tilde x \sH_i(1-z)}{u-1+z}  ,
\eeq
where we denote
  \[
    \tilde x := \frac u{(1-z)^{m+2}}   \left(1- z \frac{u ^{m+1}-1}{u -1}\right).
  \]
By   specializing~\eqref{syst-H}  to $u=1-z$, we find that
  \[
    \sH_i(1-z)= (1-z)^{m+2} \sH_{i-1}(1).
  \]
We now inject in~\eqref{Hi-new-eq} first this expression of $\sH_i(1-z)$, then the expression~\eqref{syst-H} of $\sH_i(u)$, and finally  the above expression of $\tilde x$. This proves that~\eqref{Hi-new-eq} indeed holds. Hence the claimed values of the series $\sbJ_i$ are correct.  
\end{proof}

\section{Comments and perspectives}
\label{sec:final}
\subsection{Bijections?}
\label{sec:conj} 
Obviously, this paper raises the quest for a bijective proof of Theorem~\ref{thm:numbers}. It may be possible to find inspiration in some ideas used in the bijections found by Fang for related objects~\cite{fang-bipartite,fang-bridgeless,fang-preville-enumeration,fang-trinity}. Another very suggestive guideline is the following conjecture.

\begin{conjecture}\label{conj}
  The number of  greedy $m$-Tamari intervals 
  in which the maximal element has~$n_i$ ascents of length $i$, for $i\ge 1$, including a first ascent of length $\ell$, is the number of $(m+1)$-constellations having  $n_i$ white faces of degree $(m+1)i$, for $i\ge 1$, including a white root face of degree $(m+1)\ell$.
\end{conjecture}
In the above statement, an \emm ascent, is a maximal sequence of up steps, and its length is the number of steps that it contains. Note that the size of the interval is then $n=\sum i n_i$. We have checked Conjecture~\ref{conj} for $m + n\le 10$. For instance, in Example~\ref{ex:simple}, where $m=n=2$, there are $3$ intervals where the maximal element is $u$ or $v$ and has $n_1=2$ ascents of length $1$,
and $3$ intervals where the maximal element is $w$ and has $n_2=1$ ascent of length~$2$.
Accordingly, we see on  Figure~\ref{fig:const} that we have $3$ constellations with~$2$ white faces of degree $3$, and $3$ constellations with a single white face of degree $6$.

Note that the number of $(m+1)$-constellations with $n_i$ white faces of degree $(m+1)i$ is known to be:
\[
 (m+1)m^{f-1}\frac{(mn)!}{(mn-f+2)!} \ 
\prod_{i \ge 1} \frac{1}{n_i!}{(m+1)i-1 \choose i-1}^{n_i},
\]
where $n=\sum in_i$ is the number of polygons, and $f=\sum n_i$
 the number of white faces~\cite[Thm.~2.3]{mbm-schaeffer-constellations}.

 \medskip
 Another natural question deals specifically with the case $m=1$, for which a bijection has been established between ordinary Tamari intervals of size $n$ and rooted triangulations (with no loop nor multiple edge) having $n+3$ vertices~\cite{BeBo07}. In this case the rooting consists in orienting an edge. Since greedy Tamari intervals are also ordinary intervals, one can ask  which triangulations they correspond to.  Theorem~\ref{thm:numbers} shows that they are  in bijection with $2$-constellations having $n$ polygons, and hence (via the construction of~\cite[Cor.~2.4]{mbm-schaeffer-constellations}) with Eulerian triangulations having $n+2$ vertices, but in which we now allow multiple edges. For instance, when $m=1$ and $n=2$, there are $3$ Tamari intervals, which are all greedy. The corresponding two types of triangulations (first with $5$ vertices and no multiple edge nor loop, then with only $4$ vertices but with a double edge and Eulerian), are shown below. 

 \begin{center}
   \includegraphics[height=20mm]{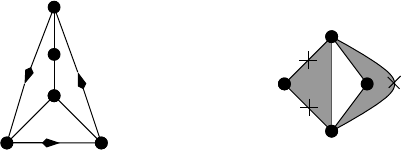}
 \end{center}

 \subsection{Other catalytic parameters?}
 
 For intervals $[v,w]$ in the \emm ordinary, $m$-Tamari lattice, the catalytic parameter considered in~\cite{mbm-fusy-preville} is not the length of the final descent of $w$, but the number of contacts in $v$. Moreover, it is easy to record as well the first ascent of $w$, and one thus discovers
 that the joint distribution of the parameters ``length of the first ascent of $w$'' and ``number of contacts of $v$, plus one'' is symmetric. A bijective proof, and a considerable refinement of this symmetry property, have then  been established in~\cite{chapoton-chatel-pons,pons-involution}.

 It is thus natural to explore, for the greedy order as well, these two statistics.

 \medskip
\noindent{\bf The first ascent of $w$.} As discussed above, the length of the first ascent of $w$ seems to be distributed like the degree of the white root face in $(m+1)$-constellations (divided by $(m+1)$). For instance, when $m=2$, the \gf\ $\widetilde \sI$ of greedy intervals counted by the size (variable $t$) and the first ascent of the upper path (variable $x$) starts
 \beq\label{ser:first-ascent}
 \widetilde \sI=  xt 
 + \left(   3 x^2 + 3 x\right) t^2 + \left(
   12 x^3 + 20 x^2 + 22 x\right) t^3 + \left(
   55 x^4 + 126 x^3 + 195 x^2 + 218 x\right) t^4 +  \LandauO(t^5),
\eeq
and it can be seen from the functional equation~\eqref{eq:const} that holds for $(m+1)$-constellations that this is also the beginning of the expansion of $\sC-1$. In particular, if we could establish that $1+\widetilde \sI$ satisfies the same equation as $\sC$, this would  at once prove and refine Theorem~\ref{thm:numbers}, without having to solve a functional equation as we did in Section~\ref{sec:sol}. Note that the functional equation~\eqref{eq:const} can be refined so as to record the degrees of (non-root) white faces; see~\cite[Thm.~4.1]{fang-these}.

\medskip

\noindent{\bf The number of contacts of $v$.} Note that a path of size $n$ has at most $n-1$ contacts, while the length of final descent can be as large as $mn$. So there is no hope to have an equidistribution of these two parameters. The length of the first ascent, on the other hand, is at most $n$, hence it could be related to the number of contacts. However, for $m=2$ again, the \gf\ counting greedy intervals with respect to the size and contacts of the lower path
starts
\[
  t+
\left(  3 x + 3\right) t^2+
\left(  9 x^2 + 23 x + 22\right) t^3+×\LandauO(t^4).
\]
Comparing with~\eqref{ser:first-ascent} shows that there is no obvious relation with the first ascent. The case  $m=1$ does not behave better.

\subsection{Labelled greedy intervals}
Another natural question deals with \emm labelled, greedy intervals. It was proved in~\cite{mbm-chapuy-preville}, again with a motivation in algebraic combinatorics, that the number of ordinary $m$-Tamari intervals $[v,w]$ of size $n$ in which the up steps of $w$ are labelled with $1, \ldots, n$  in such a way labels increase along any ascent, equals $(m+1)^n (mn+1)^{n-2}$. We have thus explored the corresponding labelled greedy intervals, but the numbers that we obtain do not seem to factor nicely. 

\subsection{A $q$-analogue}

As in the case of ordinary intervals, we can consider a $q$-analogue of our counting problem by recording, for each interval $[v,w]$, the length of the longest chain going from $v$ to $w$ in the greedy poset. It can be proved that the basic functional equation of Proposition~\ref{prop:func-eq} is modified in a very natural form:
\[ 
  x^2\sI =t(x+x^2\sI\Delta_q)^{(m+2)}(1),
\] 
where now
\[
  \Delta_q \sF(x):= \frac{\sF(xq)-\sF(1)}{xq-1},
\]
with obvious notation.


\bibliographystyle{abbrv}
\bibliography{tamar-glouton}

\appendix

\section{Polynomial identities for the operator $\nablam$}
\label{sec:app}

In this section, we establish polynomial identities involving the
operator $\nablam$ defined in~\eqref{def_nabla}, and use them to complete the proof of Theorem~\ref{thm:series}.

In view of the form~\eqref{Ji-expr}-\eqref{Hi-def} of the claimed values of the series $\hJ_i$, we introduce the following polynomials in $u$: for $a$ and $ \ell$ two nonnegative integers, let
\begin{equation}
  \label{def_R}
  R^a_{\ell}(u) = \sum_{e=0}^{a}\binom{e+\ell}{e}u^e.
\end{equation}
We extend this to $a<0$ by setting $R^{a}_\ell(u)=0$ for all $\ell$ (corresponding to the empty sum). Furthermore, considering that $R^a_{\ell}(u)$ is also a polynomial in $\ell$, we set $R^a_{-1}(u)=1$ for $a\ge 0$.

The polynomials $R^a_\ell(u)$ behave nicely with respect to the action of the operator $\nablam$.

\begin{lemma}
  \label{nabla_on_R}
 Let $\ell \geq -1$.  For $a \geq 0$, we have
  \begin{equation*}
    \nablam\left(u^{m+1-a}   R^a_{\ell}(u)\right) = -u^{m+2-a}  R^{a-1}_{\ell+1}(u).
  \end{equation*}
Moreover,  for $b \geq 1$,
  \begin{equation*}
    \nablam\left(u^{m+a+b}  R^a_{\ell}(1/u)\right) = u^{m+a+b}  R^{a+b-2}_{\ell+1}(1/u).
  \end{equation*}
\end{lemma}
\begin{proof} 
  Both equalities are proved similarly. The special case $\ell=-1$ is
  first checked separately. Now assume $\ell \geq 0$. Using the 
  definitions~\eqref{def_R} and~\eqref{def_nabla} of $R^a_{\ell}$ and
  $\nablam$, one finds a double sum. The first one involves a variable $e$ and comes from $R^a_{\ell}$, and the second one comes from the expansion of $\nablam(u^*)$. For instance, in order to prove the second identity of the lemma, we start with
  \[
    \nablam\left(u^{m+a+b}R^a_{\ell}(1/u)\right) = \sum_{e=0}^{a}\binom{e+\ell}{e}
    \sum^{m+a+b-e}_{j=m+2} u^j = u^{m+a+b}\sum_{e=0}^{a}\binom{e+\ell}{e}\sum_{f=e}^{a+b-2} (1/u)^f.
  \]
  Exchanging the summations, one
  concludes by a simple identity on binomial coefficients:
  \begin{align*}
     \nablam\left (u^{m+a+b}R_{a,\ell}(1/u)\right)& = u^{m+a+b}\sum_{f=0}^{a+b-2} (1/u)^f \sum_{e=0}^{f}\binom{e+\ell}{e}\\
                                               &= u^{m+a+b}\sum_{f=0}^{a+b-2} (1/u)^f \binom{f+\ell+1}{f}\\
    &=  u^{m+a+b} R^{a+b-2}_{\ell+1}(1/u).
  \end{align*}
\end{proof}

Let us now return to the series $\s H_i$ defined, for $0\le i \le m$, by~\eqref{Hi-def}. We can write them as:
\beq\label{H-i}
\s  H_i =  \sum_{k=0}^{i+1} (-z)^k  \left[
  \binom{i+1}{k}u^{i+1}  R^{m-i}_{k-1}(u)   +  \binom{m+k-i-1}{k-1}   u^{m+i+2-k}R^{i-k}_k(1/u)\right].
\eeq
where by convention $\binom{m-i-1}{-1}=0$. 

  Recall that the proof of Theorem~\ref{thm:series} will be complete once the following proposition is established.

\begin{prop}\label{prop:Hi-eq}
 The above series satisfy $\s H_i =  (u+z \nablam) \s H_{i-1}$  for $1 \leq i \leq m+1$.
 \end{prop}

\begin{proof} 
   We prove separately the cases $i\le m$ and $i=m+1$. Let us start with $i\le m$.
   Observe that the sums over $k$ in~\eqref{H-i} can be extended to all values $k\ge 0$: the binomial coefficient $\binom{i+1}{k}$ vanishes when $k>i+1$,
   and the sum $R^{i-k}_k$ is empty as soon as $i>k$.  We form the polynomial $\s H_i -  (u+z \nablam) \s H_{i-1}$ and extract the coefficient of $(-z)^k$, with $0\le k\le i+1$ (this coefficient being obviously zero for larger values of $k$). The coefficient  of $(-z)^k$ reads $c_++c_-$, with
  \[
    c_+=\binom{i+1}{k}u^{i+1}  R^{m-i}_{k-1}(u)- u \binom{i}{k}u^{i}  R^{m-i+1}_{k-1}(u)+\binom{i}{k-1} \nablam\left(u^{i} R^{m-i+1}_{k-2}(u)\right),
  \]
  \begin{multline*}
    c_-=  \binom{m+k-i-1}{k-1}   u^{m+i+2-k}R^{i-k}_k(1/u)-
    u\binom{m+k-i}{k-1}   u^{m+i+1-k}R^{i-k-1}_k(1/u)\\
    + \binom{m+k-i-1}{k-2} \nablam\left(  u^{m+i+2-k}R^{i-k}_{k-1}(1/u)\right),
  \end{multline*}
  where  the binomial coefficients $\binom{a}{b}$ are zero when $b<0$. For $k=0$, the term $c_+$ vanishes, since $R^a_{-1}(u)$ has been defined to be $1$ for $a\ge 0$. The term $c_-$ vanishes as well when $k=0$, thus the polynomial $\s H_i -  (u+z \nablam) \s H_{i-1}$ has no constant term.

  So let us take $k\in \llbracket1, i+1\rrbracket$ and examine the term $c_+$. Using Lemma~\ref{nabla_on_R}, we can reexpress the term involving $\nablam$, and we thus obtain:
  \begin{align*}
    c_+ &= u^{i+1}\left[\binom{i+1}{k}  R^{m-i}_{k-1}(u)-  \binom{i}{k}  R^{m-i+1}_{k-1}(u)-\binom{i}{k-1} R^{m-i}_{k-1}(u)\right]\\
        &= u^{i+1}  \binom{i}{k} \left[ R^{m-i}_{k-1}(u)- R^{m-i+1}_{k-1}(u)\right]
          \hskip 21mm \text{by Pascal's formula,}\\
        &= -u^{i+1}  \binom{i}{k} \binom{m-i+k}{k-1} u^{m-i+1} \hskip 20mm \text{by definition of } R^a_{k-1}(u),\\
    &= - \binom{i}{k} \binom{m-i+k}{k-1} u^{m+2}.
  \end{align*}
  In particular $c_+$ is zero when $k=i+1$. The same holds for $c_-$ in this case, since all sums  $R^a_*$ involved in its expression are empty.

  So let us finally consider the expression of $c_-$ for $k \in \llbracket 1, i\rrbracket$. We now use the second part of  Lemma~\ref{nabla_on_R} and obtain
  \begin{align*}
    c_- &=  u^{m+i+2-k}\left[ \binom{m+k-i-1}{k-1} R^{i-k}_k(1/u)-
    \binom{m+k-i}{k-1} R^{i-k-1}_k(1/u)
          + \binom{m+k-i-1}{k-2} R^{i-k}_{k}(1/u)\right]\\
        &= u^{m+i+2-k} \binom{m+k-i}{k-1} \left[ R^{i-k}_k(1/u)-R^{i-k-1}_k(1/u)\right]\\
        &= u^{m+i+2-k} \binom{m+k-i}{k-1} \binom{i}{k} u^{-(i-k)}\\
    &= \binom{m+k-i}{k-1} \binom{i}{k}u^{m+2} .
  \end{align*}
  Comparing with the expression of $c_+$ shows that  $\s H_i -  (u+z \nablam) \s H_{i-1}$ is zero for $1\le i \le m$.
  \medskip

  Let us finally prove that $\s H_{m+1}=  (u+z \nablam) \s H_{m}$. The  case $i=m$ of~\eqref{H-i} gives
  \[
    \s  H_m = u^{m+1} (1-z)^{m+1}  +\sum_{k=1}^{m} (-z)^k 
    u^{2m+2-k}R^{m-k}_k(1/u).
    \]
   We note that  $\nablam(u^{m+1})=0$ and use again the second part of  Lemma~\ref{nabla_on_R}. This gives
    \begin{align*}
\allowdisplaybreaks   (u+z \nablam) \s H_{m} &=     u^{m+2}(1-z)^{m+1} + \sum_{k=1}^{m} (-z)^k 
  \left[ u^{2m+3-k}R^{m-k}_k(1/u) +z
                           u^{2m+2-k}R^{m-k}_{k+1}(1/u)\right]\\
                         &= u^{m+2}(1-z)^{m+1} + \sum_{k=1}^{m} (-z)^k  u^{2m+3-k}R^{m-k}_k(1/u)
                             - \sum_{k=2}^{m+1} (-z)^{k}  u^{2m+3-k}R^{m-k+1}_{k}(1/u)\\
                         &= u^{m+2}(1-z)^{m+1} + (-z)u^{2m+2} R^{m-1}_1(1/u)\\
      & \hskip 50mm + \sum_{k=2}^{m+1} (-z)^ku^{2m+3-k}\left[ R^{m-k}_k(1/u)-R^{m-k+1}_{k}(1/u)\right].
    \end{align*}
    The rest of the calculation is straightforward: using the definition of $ R^{m-1}_1(1/u)$ and
    \[
      R^{m-k}_k(1/u)-R^{m-k+1}_{k}(1/u)=- \binom{m+1}{k} u^{-(m-k+1)},
    \]
    one finally recovers the expression~\eqref{H-m+1} of $\s H_{m+1}$. This concludes the proof of Proposition~\ref{prop:Hi-eq} and Theorem~\ref{thm:series}.
 \end{proof}

\end{document}